\documentclass[reqno]{amsart}

\usepackage{amsmath,amsthm,amssymb,amscd}
 \newtheorem{thm}{Theorem}[section]
 \newtheorem{cor}[thm]{Corollary}
 \newtheorem{lem}[thm]{Lemma}
 
 \theoremstyle{definition}
 
 \theoremstyle{remark}
 \newtheorem{rem}[thm]{Remark}
 
 \numberwithin{equation}{section}

 \makeatletter
\@namedef{subjclassname@2020}{%
  \textup{2020} Mathematics Subject Classification}
\makeatother
\begin{document}

%
%
%
%
%
%
%
%


\title[$p$-adic quotient sets: diagonal forms]
 {$p$-adic quotient sets: diagonal forms}

\author[Deepa Antony]{Deepa Antony}

\address{Department of Mathematics \\
	Indian Institute of Technology Guwahati \\
	Assam, India, PIN- 781039}

\email{deepa172123009@iitg.ac.in}

\author{Rupam Barman}
\address{Department of Mathematics \\
	Indian Institute of Technology Guwahati \\
	Assam, India, PIN- 781039}
\email{rupam@iitg.ac.in}

\author{Piotr Miska}
\address{Institute of Mathematics\\ Faculty of Mathematics and Computer Science\\ Jagiellonian University\\ Krak\'{o}w, Poland}
\email{piotr.miska@uj.edu.pl}

\thanks{The research of the third author was partially supported by the grant of the Polish National Science Centre no. UMO-2018/29/N/ST1/00470}

\date{January 12, 2022}

\begin{abstract}
	For a set of integers $A$, we consider $R(A)=\{a/b: a, b\in A, b\neq 0\}$. It is an open problem to study the denseness of $R(A)$ in the $p$-adic numbers when $A$ is the set of nonzero values attained by an integral form. This problem has been answered for quadratic forms. Very recently, Antony and Barman have studied this problem for the diagonal binary cubic forms $ax^3+by^3$, where $a$ and $b$ are integers. In this article, we study this problem for diagonal forms. We extend the results of Antony and Barman to the diagonal binary forms $ax^n+by^n$ for all $n\geq 3$. We also study $p$-adic denseness of quotients of nonzero values attained by diagonal forms of degree $n\geq 3$, where $\gcd(n,p(p-1))=1$.
\end{abstract}
\maketitle
\section{Introduction and statement of results} 
For a set of integers $A$, the set $R(A)=\{a/b:a,b\in A, b\neq 0\}$ is called the ratio set or quotient set of $A$. The major open problem is to characterise all subsets of $\mathbb{N}$ whose ratio sets are dense in the positive real numbers. In this direction, many results have already been obtained, see for example \cite{real-1, real-2, real-3, real-4, real-5, real-6, real-7, real-8, real-9, real-10, real-11, real-12, real-13, real-14}.  An analogous study has also been done for algebraic number fields, see for example \cite{jaitra, algebraic-1, algebraic-2}. 
\par In recent years, this problem has also been considered in the $p$-adic set-up, and many interesting results have already been showed \cite{cubic, Donnay, garciaetal, garcia-luca, miska, piotr, miska-sanna,  Sanna}. For a prime $p$, let $\mathbb{Q}_p$  denote the field of $p$-adic numbers. In \cite{garcia-luca}, Garcia and Luca initiated a study of denseness of the sets $R(A)$, $A\subset\mathbb{N}$, in $p$-adic numbers. They showed that the set of quotients of Fibonacci numbers is dense in $\mathbb{Q}_p$ for all primes $p$. Then Sanna \cite{Sanna} showed that the fractions of $k$-generalised Fibonacci numbers are dense in $\mathbb{Q}_p$. Later Garcia, Hong, Luca, Pinsker, Sanna, Schechter and Starr \cite{garciaetal} further broadened such investigation by considering different kinds of sets.
In \cite{piotr}, Miska, Murru and Sanna studied the denseness of $R(A)$ in $\mathbb{Q}_p$ when $A$ is the image of $\mathbb{N}$ under a polynomial $f\in\mathbb{Z}[X]$ or $A$ is the set of sums of $m$ $n$th powers of integers, where $m,n\in\mathbb{N}$ are fixed. 
\par 
A form of degree $n$ is a homogeneous polynomial
\begin{align*}
F(x_1,x_2, \ldots, x_r)=\sum_{1\leq i_1\leq i_2\leq \ldots\leq i_n\leq r}a_{i_1i_2\ldots i_n}x_{i_1} x_{i_2} \cdots x_{i_n}. 
\end{align*} 
We say that $F$ is integral if $a_{i_1i_2\ldots i_n} \in \mathbb{Z}$ for all $(i_1,i_2,\ldots,i_n)$ and $F$ is primitive if there is no positive integer $d>1$ such that $d\mid  a_{i_1i_2\ldots i_n} \in \mathbb{Z}$ for all $(i_1,i_2,\ldots,i_n)$. 
A form $F$ is said to be isotropic over a field $\mathbb{F}$ if there is a nonzero  vector $\overline{x}\in \mathbb{F}^r$ such that $F(\overline{x}) =0$. Otherwise  $F$ is said to be anisotropic over $\mathbb{F}$. The ratio set generated by an integral form $F$ is 
\begin{align*}
R(F)=\{F(\overline{x})/F(\overline{y}): \overline{x}, \overline{y}\in \mathbb{Z}^r, F(\overline{y})\neq 0\}.
\end{align*} 
It is an open problem to study the denseness of $R(F)$ in the $p$-adic numbers. This problem has been completely answered by Donnay, Garcia and Rouse \cite{Donnay} for quadratic forms. They showed that for a non-singular binary quadratic form $Q$, $R(Q)$ is dense in $\mathbb{Q}_p$ if and only if the discriminant of $Q$ is a nonzero square in $\mathbb{Q}_p$. They also proved that for a non-singular quadratic form $Q$ in at least three variables, $R(Q)$ is always dense in $\mathbb{Q}_p$. Later Miska \cite{miska} gave shorter proof for the same.
\par In a very recent article \cite{cubic}, Antony and Barman have studied the denseness of $R(C)$ in the $p$-adic numbers when $C$ is the diagonal cubic form $C(x,y)=ax^3+by^3$, where $a$ and $b$ are nonzero integers. In this article, we study the denseness of $R(F)$ in the $p$-adic numbers when $F$ is an integral diagonal form. Firstly, we extend the results of \cite{cubic} to the diagonal binary form $F(x,y) = ax^n+by^n$ for all $n\geq 3$.
To be specific, we prove the following result which extends \cite[Theorem 1.2]{cubic}.
For a prime $p$, let $\nu_p$ denote the $p$-adic valuation function. Moreover, for each logic sentence $T$ we put
$$[T]=\begin{cases}1\mbox{ if }T\mbox{ is true},\\0\mbox{ if }T\mbox{ is false}.\end{cases}$$

\begin{thm}\label{thm-3}
	Let $n\geq 3$. Let $F(x,y) = ax^n+by^n$ be integral.
	\begin{enumerate}
		\item If $p\nmid ab$, then $R(F)$ is dense in $\mathbb{Q}_p$ if and only if $-a^{-1}b$ is an $n$th power residue modulo $p^{\nu_p(n)+\nu_p\left(2^{[2\mid n]}\right)+1}$. 
		\item If  $a=p^k\ell$ such that $p\nmid \ell$ and  $n\mid k$, then $R(F)$ is dense in $\mathbb{Q}_p$ if and only if $-b^{-1}\ell$ is an $n$th power residue modulo $p^{\nu_p(n)+\nu_p\left(2^{[2\mid n]}\right)+1}$.
		\item If $a=p^k\ell, p\nmid \ell$ and $n\nmid k$, then $R(F)$ is not dense in $\mathbb{Q}_p$. 
	\end{enumerate}
\end{thm}
We remark that if $n$ is an odd integer, then $-a^{-1}b$ is an $n$th power residue modulo $p^{\nu_p(n)+\nu_p\left(2^{[2\mid n]}\right)+1}$ if and only if $a^{-1}b$ is an $n$th power residue modulo $p^{\nu_p(n)+1}$.
\par 
Let us note that Theorem \ref{thm-3} simplifies if $\gcd(n, p(p-1))=1$.
\begin{cor}\label{cor-new}
Let $n\geq 3$ and $p$ be a prime number such that $\gcd(n, p(p-1))=1$.
\begin{enumerate}
\item Let $F(x,y) = ax^n+by^n$ be integral with $ab\neq 0$ and $\nu_p(a)\equiv\nu_p(b)\pmod{n}$. Then $R(F)$ is dense in $\mathbb{Q}_p$.
\item Let $F(x,y) = a_1x_1^n+a_2x_2^n+\cdots +a_rx_r^n$ be integral with $a_1a_2\cdots a_r\neq 0$ and such that $\nu_p(a_i)\equiv\nu_p(a_j)\pmod{n}$ for some $1\leq i<j\leq r$. Then $R(F)$ is dense in $\mathbb{Q}_p$.
\end{enumerate}
\end{cor}
The second part of the above corollary shows that if $\gcd(n, p(p-1))=1$ and $F$ is an integral diagonal form of degree $n$ depending essentially on at least $n+1$ variables, then $R(F)$ is dense in $\mathbb{Q}_p$. However, this result is not optimal in the sense that $n+1$ is not the least number $r$ such that for every integral diagonal form $F$ of degree $n$ depending essentially on at least $r$ variables the set $R(F)$ is dense in $\mathbb{Q}_p$. This is due to the result below.
\begin{thm}\label{thm-new}
Let $n\geq 3$ and $p$ be a prime number such that $\gcd(n, p(p-1))=1$. Let $F(x,y) = a_1x_1^n+a_2x_2^n+\cdots +a_rx_r^n$ be integral and such that $r>\frac{n}{2}$, $a_1a_2\cdots a_r\neq 0$ and $\nu_p(a_i)\not\equiv\nu_p(a_j)\pmod{n}$ for any $1\leq i<j\leq r$ (in particular $r\leq n$). Then $R(F)$ is dense in $\mathbb{Q}_p$.
\end{thm}
As an immediate consequence of Corollary \ref{cor-new}(2) and Theorem \ref{thm-new} we get the following.
\begin{cor}\label{cor-new2}
Let $n\geq 3$ and $p$ be a prime number such that $\gcd(n, p(p-1))=1$. The number $\left\lfloor\frac{n}{2}\right\rfloor +1=\left\lceil\frac{n+1}{2}\right\rceil$ is the least number $r$ such that for every integral diagonal form $F$ of degree $n$ depending essentially on at least $r$ variables the set $R(F)$ is dense in $\mathbb{Q}_p$.
\end{cor}
Next, we give a sufficient condition for the denseness of the ratio set of the values attained by a diagonal cubic form. In \cite[Theorem 1.8]{piotr}, Miska, Murru and Sanna proved that the ratio set of $\{x_1^3+\cdots + x_r^3: x_1, \ldots, x_r\in \mathbb{Z}_{\geq 0}\}$ is dense in $\mathbb{Q}_p$ for all $r\geq 2$. The following results also generalise this fact.
\begin{thm}\label{thm-5}
 Let $C(x,y,z)=ax^3+by^3+cz^3$ be integral. If $p\neq3$ and $p\nmid abc$, then $R(C)$ is dense in $\mathbb{Q}_p.$
 \end{thm}
Using Theorem \ref{thm-5}, we have the following corollary.
\begin{cor}\label{cor-1}
	Let $C(x_1,x_2, \ldots, x_r)=a_1x_1^3+a_2x_2^3+\cdots +a_rx_r^3$ be integral. Suppose that $p\neq3$ and $\nu_p(a_i)\equiv \nu_p(a_j)\equiv \nu_p(a_k)\pmod{3}$ for some $i, j,k\in \{1, \ldots, r\}$, $i<j<k$. 
	Then $R(C)$ is dense in $\mathbb{Q}_p$.
\end{cor}
We also have the following corollary which follows from Corollary \ref{cor-1}.
\begin{cor}\label{cor-2}
For each $p\neq 3$ and every integral diagonal cubic form \linebreak$C(x_1,x_2,\ldots ,x_r) = a_1x_1^3+a_2x_2^3+\cdots +a_rx_r^3$ with $a_1a_2\cdots a_r\neq 0$ and $r\geq 7$ the set $R(C)$ is dense in $\mathbb{Q}_p$.
\end{cor}
Studying denseness of $R(F)$ when $F$ is any integral form seems to be a difficult problem. Our last result is on any form $F$ which is anisotropic modulo $p$, and this extends Theorem 1.1 of \cite{cubic}. To be specific, we prove the following result.
\begin{thm}\label{thm-4}
	Let $F$ be any form of degree $n\geq2$ that is primitive and integral. If $F$ is anisotropic modulo $p$, then $R(F)$ is not dense in $\mathbb{Q}_p.$
\end{thm}
 \section{Preliminaries}
For a prime number $p$, every nonzero rational number $r$ has a unique representation of the form $r= \pm p^k a/b$, where $k\in \mathbb{Z}, a, b \in \mathbb{N}$ and $\gcd(a,p)= \gcd(p,b)=\gcd(a,b)=1$. 
The $p$-adic valuation of such an $r$ is $\nu_p(r)=k$ and its $p$-adic absolute value is $\|r\|_p=p^{-k}$. By convention, $\nu_p(0)=\infty$ and $\|0\|_p=0$. The $p$-adic metric on $\mathbb{Q}$ is given by the formula $d(x,y)=\|x-y\|_p$. 
The field $\mathbb{Q}_p$ of $p$-adic numbers is the completion of $\mathbb{Q}$ with respect to the $p$-adic metric. We denote by $\mathbb{Z}_p$ the ring of $p$-adic integers which is the set of elements of $\mathbb{Q}_p$ with $p$-adic norm less than or equal to 1. The unit group of $\mathbb{Z}_p$ is denoted by $\mathbb{Z}_p^{\times}$.
\par We next state a few results which will be used in the proofs of our theorems.
\begin{lem}\label{lem2}\cite[Lemma 2.1]{garciaetal}
	If $S$ is dense in $\mathbb{Q}_p$, then for each finite value of the $p$-adic valuation, there is an element of $S$ with that valuation.
\end{lem}
 \begin{thm}\cite[Hensel's lemma]{cassels}
 Let $f(x)\in \mathbb{Z}_p[x]$ and $a\in \mathbb{Z}_p$ satisfy $\|f(a)\|_p<\|f'(a)\|_p^2$.  Then there exists an $\alpha \in \mathbb{Z}_p$ such that $f(\alpha)=0$.
 \end{thm}
\begin{thm}\label{thm6}\cite[Corollary 1.3]{piotr}
	Let $f\colon\mathbb{Z}_p\rightarrow \mathbb{Q}_p$ be an analytic function with a simple zero in $\mathbb{Z}_p$. Then  $R(f(\mathbb{N}))$ is dense in $\mathbb{Q}_p$.
\end{thm}
Recall that a vector $\overline{x}\in \mathbb{F}^r$ such that $F(\overline{x})=0$ is a 
non-singular zero of $F$ if $\frac{\partial F}{\partial x_i}(\overline{x})\neq 0$ for some $i\in \{1, \ldots, r\}$.
We use the following lemma due to Heath-Brown to prove Theorem \ref{thm-5}. This lemma guarantees the existence of a non-singular zero for a particular cubic form modulo $p$.
\begin{lem}\label{lem-7}\cite[Lemma 6]{heath}
	Let $p\neq3$ and suppose that $$f(x,y,z)=ax^3+bxy^2+cy^3+(dx+ey)z^2+fz^3 \in \mathbb{F}_p[x,y,z],$$ with $acf\neq0$. Then $f$ has at least one non-singular zero over $\mathbb{F}_p$, where $\mathbb{F}_p$ is the field with $p$ elements.
\end{lem}
\section{Proofs of the results}
We first prove a lemma which will be used to prove Theorem \ref{thm-3}.
\begin{lem}\label{lem-1} Let $p$ be a prime, $k\geq 1$ and $u\in \mathbb{Z}_p^\times$. Then $u$ is a $p^k$th power in $\mathbb{Z}_p$ if and only if $u$ is a $p^k$th power residue modulo $p^{k+\nu_p(2)+1}$ .
\end{lem}

\begin{proof}  We first prove for any non-negative integer $\ell$ that if  $u\in \mathbb{Z}_p^\times $ is a $p^k$th power residue modulo $p^{k+\nu_p(2)+\ell+1}$, then $u$ is a $p^k$th power residue modulo $p^{k+\nu_p(2)+\ell+2}$. Suppose that $u\equiv a^{p^k}\pmod{p^{k+\nu_p(2)+\ell+1}}$ for some $a\in \mathbb{Z}_p.$ Then $a\in \mathbb{Z}_p^\times$ and $u/a^{p^k}\equiv 1\pmod{p^{k+\nu_p(2)+\ell+1}}$. Hence $$u/a^{p^k}\equiv 1+cp^{k+\nu_p(2)+\ell+1} \pmod{p^{k+\nu_p(2)+\ell+2}},$$ where $0\leq c \leq p-1$. We have $$\left(1+cp^{\nu_p(2)+\ell+1}\right)^{p^k}\equiv 1+cp^{k+\nu_p(2)+\ell+1} \pmod{p^{k+\nu_p(2)+\ell+2}}.$$ Therefore, $$u\equiv a^{p^k}\left(1+cp^{\nu_p(2)+\ell+1}\right)^{p^k} \pmod{p^{k+\nu_p(2)+\ell+2}}.$$
Thus, $u$ is a $p^k$th power residue modulo $p^{k+\nu_p(2)+\ell+2}$.
\par 
Let $u\in \mathbb{Z}_p^\times$. Suppose that $u$ is a $p^k$th power residue modulo $p^{k+\nu_p(2)+1}$.	Then, using the above fact repeatedly we find that $u$ is a $p^k$th power residue modulo $p^{2k+1}$. Hence, by Hensel's lemma, $u$ is a $p^k$th power in $\mathbb{Z}_p$. Conversely, suppose that $u$ is a $p^k$th power in $\mathbb{Z}_p$. Since $u\in \mathbb{Z}_p^\times$, we infer that $u$ is a $p^k$th power residue modulo $p^{k+\nu_p(2)+1}$. This completes the proof of the lemma.
\end{proof}

Having Lemma \ref{lem-1} showed, we are ready to prove Theorem \ref{thm-3}.

\begin{proof}[Proof of Theorem \ref{thm-3}]
We first prove part $(1)$ of the theorem.
Let $n=p^km$ where $p\nmid m$.
Suppose that $-a^{-1}b$ is an $n$th power residue modulo $p^{\nu_p(n)+\nu_p\left(2^{[2\mid n]}\right)+1}$. Then, $-a^{-1}b$ is a $p^k$th power residue modulo $p^{\nu_p(n)+\nu_p\left(2^{[2\mid n]}\right)+1}$. By Lemma \ref{lem-1}, $-a^{-1}b$ is a $p^k$th power in $\mathbb{Z}_p$. Also, by Hensel's lemma, $-a^{-1}b$ is an $m$th power in $\mathbb{Z}_p$. Therefore, $-a^{-1}b$ is an $n$th power in $\mathbb{Z}_p$. Thus $-a^{-1}b=u^n$ for some $u\in \mathbb{Z}_p$. This implies that $u$ is a root of the polynomial $f(x)=x^n+a^{-1}b$. Clearly, $f'(u)=nu^{n-1}\neq 0$, and hence the polynomial $f(x)=x^n+a^{-1}b$ has a simple root in $\mathbb{Z}_p$. Therefore, by Lemma \ref{thm6}, $R(f(\mathbb{N}))\subset R(F)$ is dense in $\mathbb{Q}_p$. Hence $R(F)$ is dense in $\mathbb{Q}_p$. Conversely, the denseness of $R(F)$ in $\mathbb{Q}_p$ implies that $-a^{-1}b$ is an $n$th power in $\mathbb{Z}_p$. Therefore, $-a^{-1}b$ is an $n$th power residue modulo $p^{\nu_p(n)+\nu_p\left(2^{[2\mid n]}\right)+1}$. 
\par We next prove part (2) of the theorem. Let $a=p^k\ell$, where $p\nmid \ell$ and $n|k$. We write $k=nk'$.
Suppose that $-b^{-1}\ell$ is not an $n$th power residue modulo $p^{\nu_p(n)+\nu_p\left(2^{[2\mid n]}\right)+1}$. Let $\widetilde{F}(x,y)=b^{-1}\ell x^n+y^n$. By the first part of the theorem, $R(\widetilde{F})$ is not dense in $\mathbb{Q}_p$. We have $b^{-1}F(x, y)=p^kb^{-1}\ell x^n+y^n=b^{-1}\ell(p^{k^\prime}x)^n+y^n=\widetilde{F}(p^{k^\prime}x, y)$. 
Since $R(F)=R(b^{-1}F)\subset R(\widetilde{F})$, $R(F)$ is not dense in $\mathbb{Q}_p$.
 \par 
Conversely, suppose that $-b^{-1}\ell$ is an $n$th power residue modulo $p^{\nu_p(n)+\nu_p\left(2^{[2\mid n]}\right)+1}$. We have $b^{-1}F(x,y)=b^{-1}\ell p^{nk^\prime}x^n+y^n$. Using part (1) of the theorem, we have that $R(\widetilde{F})$ is dense in $\mathbb{Q}_p$, where  $\widetilde{F}(x,y)=b^{-1}\ell x^n+y^n$. Since 
 \begin{align*}
\frac{\widetilde{F}(x,y)}{\widetilde{F}(z,w)}=\frac{p^{nk^\prime} \widetilde{F}(x,y)}{p^{nk^\prime}\widetilde{F}(z,w)}
=\frac{\widetilde{F}(p^{k^\prime}x,p^{k^\prime}y)}{\widetilde{F}(p^{k^\prime}z,p^{k^\prime}w)}
=\frac{F(x,p^{k^\prime}y)}{F(z,p^{k^\prime}w)},
 \end{align*}
 therefore $R(F)$ is dense in $\mathbb{Q}_p$. This completes the proof of part (2) of the theorem.
 \par Finally, we prove part (3) of the theorem. Here $a=p^k\ell, p\nmid \ell$ and $n\nmid k$. Suppose that $R(F)$ is dense in $\mathbb{Q}_p$. Let $\widetilde{F}(x, y)=b^{-1}F(x, y)$. Then $R(\widetilde{F})$ is dense in $\mathbb{Q}_p$. 
 Choose  an $m$ which is not an $n$th power residue modulo $p$. There exist $x,y,z,w \in \mathbb{Z}$ not all multiples of $p$ such that
 	\begin{align*}
 	 \left\Vert\frac{\widetilde{F}(x,y)}{\widetilde{F}(z,w)}-m\right\Vert_p < \frac{1}{p^k}.
 	\end{align*}
 	This yields
 	\begin{align*} 
 	 \left\Vert y^n-mw^n+p^kb^{-1}\ell (x^n-mz^n) \right\Vert_p =\left\Vert \widetilde{F}(x,y)-m\widetilde{F}(z,w)\right\Vert_p<&\ \frac{\left\Vert \widetilde{F}(z,w)\right\Vert_p}{p^k}\leq  \frac{1}{p^k}.
  	\end{align*}
  	If $p\nmid y$ or $p\nmid w$, then $y^n-mw^n \not\equiv 0 \pmod{p}$ since $m$ is not an $n$th power residue modulo $p$. Hence $\left\Vert \widetilde{F}(x,y)-m\widetilde{F}(z,w)\right\Vert_p=1$ which is a contradiction. 
  	Therefore $p\mid y$ and $p\mid w$, which gives $\nu_p(y^n-mw^n)=nt$, where $t$ is a positive integer. Since $p\mid y$ and $p\mid w$, we have either $p\nmid x$ or $p\nmid z$. This yields $\nu_p(x^n-mz^n)=0$. Since $n\nmid k$, we have that $\widetilde{F}(x, y)-m\widetilde{F}(z, w)$ is the sum of a $p$-adic integer with valuation being a positive multiple of $n$ and a $p$-adic integer with valuation equal to $k$.
  	Hence $\left\Vert \widetilde{F}(x,y)-m\widetilde{F}(z,w)\right\Vert_p \geq p^{-k}$, which gives a contradiction. Consequently, $R(\widetilde{F})$ is not dense in $\mathbb{Q}_p$. Finally, $R(F)$ is not dense in $\mathbb{Q}_p$.
  	This completes the proof of the theorem.
 \end{proof} 
 
\begin{proof}[Proof of Corollary \ref{cor-new}]
 \par We start with the proof of part $(1)$. Let $a=p^{k_1}\ell_1$ and $b=p^{k_2}\ell_2$, where $k_1\equiv k_2\pmod{n}$ and $p\nmid\ell_1\ell_2$. Without loss of generality assume that $k_1\geq k_2$. Let $\widetilde{F}(x,y)=p^{k_1-k_2}\ell_1x^n+\ell_2y^n$. Since $\gcd(n, p(p-1))=1$, every integer is an $n$th power residue modulo $p$. Therefore we can apply Theorem \ref{thm-3}(2) to conclude that $R(\widetilde{F})$ is dense in $\mathbb{Q}_p$. At last, we see that  
 \begin{align*}
\frac{\widetilde{F}(x,y)}{\widetilde{F}(z,w)}=\frac{p^{k_2} \widetilde{F}(x,y)}{p^{k_2}\widetilde{F}(z,w)}
=\frac{p^{k_1}\ell_1x^n+p^{k_2}\ell_2y^n}{p^{k_1}\ell_1z^n+p^{k_2}\ell_2w^n}
=\frac{F(x,y)}{F(z,w)},
 \end{align*}
which means that $R(F)=R(\widetilde{F})$ is dense in $\mathbb{Q}_p$.
 \par For the proof of (2), we use (1) to claim that $R(a_ix_i^n+a_jx_j^n)$ is dense in $\mathbb{Q}_p$. Since $R(a_ix_i^n+a_jx_j^n)\subset R(F)$, we see that $R(F)$ is dense in $\mathbb{Q}_p$.
\end{proof}

\begin{proof}[Proof of Theorem \ref{thm-new}]
	At first, recall that $\mathbb{Z}$ is dense in $\mathbb{Z}_p$ and $F$ and the operation of division are continuous in $\mathbb{Q}_p$. This implies that $R(F)=R(F(\mathbb{Z}^r))$ is dense in $R(F(\mathbb{Z}_p^r))$. Hence, it suffices to prove that $R(F(\mathbb{Z}_p^r))=\mathbb{Q}_p$.
	\par 
	Since $\gcd(n,p-1)=1$, every $c\in\mathbb{Z}_p^\times$ is an $n$-th power modulo $p$. In other words, the polynomial $f(x)=x^n-c$ has a root $u$ modulo $p$. Then $u\in\mathbb{Z}_p^\times$. Since $\gcd(n,p)=1$, we have $f'(u)=nu^{n-1}\neq 0\pmod{p}$ and by Hensel's lemma we conclude that $f$ has a root in $\mathbb{Z}_p$, i.e. $c$ is an $n$-th power in $\mathbb{Z}_p$. Since the set $\mathbb{Z}_p^\times$ is contained in the set of $n$-th powers in $\mathbb{Z}_p$, the set of $n$-th powers in $\mathbb{Z}_p$ is exactly the set of $p$-adic integers with $p$-adic valuation divisible by $n$.
	\par 
	Next, let $d\in\mathbb{Q}_p$ be arbitrary. We will show that $d=\frac{a_i}{a_j}p^{kn}u^n$ for some $i,j\in\{1,2,\ldots, r\}$, $k\in\mathbb{Z}$ and $u\in\mathbb{Z}_p^\times$.  Consider the sets
	\begin{align*}
	\{\nu_p(a_i)\pmod{n}: &\ i\in\{1,2,\ldots, r\}\},\\
	\{(\nu_p(d)+\nu_p(a_j))\pmod{n}: &\ j\in\{1,2,\ldots, r\}\}.
	\end{align*}
	Both of them have cardinality $r>\frac{n}{2}$. Hence, they have non-empty intersection, i.e. $\nu_p(d)+\nu_p(a_j)\equiv\nu_p(a_i)\pmod{n}$ for some $i,j\in\{1,2,\ldots ,r\}$. Thus $n\mid\nu_p\left(d\frac{a_j}{a_i}\right)$, in other words $d\frac{a_j}{a_i}=p^{kn}c$ for some $k\in\mathbb{Z}$ and $c\in\mathbb{Z}_p^\times$. We already know that $c=u^n$ for some $u\in\mathbb{Z}_p^\times$. As a result, $d\frac{a_j}{a_i}=p^{kn}u^n$, or equivalently, $d=\frac{a_i}{a_j}p^{kn}u^n$. Finally, we can write
	\begin{align*}
	d=
	\begin{cases}
	\frac{a_i(p^ku)^n}{a_j}, &\ \text{ if } k\geq 0,\\
	\frac{a_iu^n}{a_jp^{-kn}}, &\ \text{ if } k<0.
	\end{cases}
	\end{align*}
	This means that $d\in R(F(\mathbb{Z}_p^r))$.
	\par 
	We showed that $R(F(\mathbb{Z}_p^r))=\mathbb{Q}_p$, which means that $R(F)$ is dense in $\mathbb{Q}_p$.
\end{proof}

\begin{proof}[Proof of Corollary \ref{cor-new2}]
	Corollary \ref{cor-new} and Theorem \ref{thm-new} show that if $F$ is an integral diagonal form of degree $n$ depending essentially on more than $\frac{n}{2}$ variables, then $R(F)$ is dense in $\mathbb{Q}_p$. It suffices to find an integral diagonal form $F_0$ of degree $n$ depending on $\left\lfloor\frac{n}{2}\right\rfloor$ variables such that $R(F_0)$ is not dense in $\mathbb{Q}_p$.
	\par 
	Consider $$F_0(x_1,x_2,x_3,\ldots ,x_{\left\lfloor\frac{n}{2}\right\rfloor})=x_1^n+px_2^n+p^2x_3^n+\cdots +p^{\left\lfloor\frac{n}{2}\right\rfloor -1}x_{\left\lfloor\frac{n}{2}\right\rfloor}^n.$$ Then, for each $x=\left(x_1,x_2,x_3,\ldots ,x_{\left\lfloor\frac{n}{2}\right\rfloor}\right)\in\mathbb{Z}^{\left\lfloor\frac{n}{2}\right\rfloor}$ we have 
	\begin{align*}
	\nu_p(F(x)) &\ =\min_{1\leq i\leq \left\lfloor\frac{n}{2}\right\rfloor} \nu_p(p^{i-1}x_i^n)\\
	&\ =\min_{1\leq i\leq \left\lfloor\frac{n}{2}\right\rfloor} (i-1+n\nu_p(x_i))\in\left\{0,1,2,\ldots, \left\lfloor\frac{n}{2}\right\rfloor -1\right\}\pmod{n}.
	\end{align*}
	Hence, $\nu_p\left(\frac{F(x)}{F(y)}\right)\not\equiv\left\lfloor\frac{n}{2}\right\rfloor\pmod{n}$ for any $x,y\in\mathbb{Z}^{\left\lfloor\frac{n}{2}\right\rfloor}$. By Lemma \ref{lem2}, $R(F_0)$ is not dense in $\mathbb{Q}_p$.
\end{proof}

\begin{rem}
{\rm Note that there exists an integral diagonal form $F$ of degree $n$ depending on $\left\lfloor\frac{n}{2}\right\rfloor$ variables such that $R(F)$ is dense in $\mathbb{Q}_p$.
\par 
Consider $$F(x_1,x_2,x_3,\ldots ,x_{\left\lfloor\frac{n}{2}\right\rfloor -1},x_{\left\lfloor\frac{n}{2}\right\rfloor})=x_1^n+px_2^n+p^2x_3^n+\cdots +p^{\left\lfloor\frac{n}{2}\right\rfloor -2}x_{\left\lfloor\frac{n}{2}\right\rfloor -1}^n+p^{\left\lfloor\frac{n}{2}\right\rfloor}x_{\left\lfloor\frac{n}{2}\right\rfloor}^n.$$ Then, for each $x=\left(x_1,x_2,x_3,\ldots ,x_{\left\lfloor\frac{n}{2}\right\rfloor}\right)\in\mathbb{Z}^{\left\lfloor\frac{n}{2}\right\rfloor}$ we have 
	\begin{align*}
	\nu_p(F(x))\in\left\{0,1,2,\ldots, \left\lfloor\frac{n}{2}\right\rfloor -2,\left\lfloor\frac{n}{2}\right\rfloor\right\}\pmod{n}.
	\end{align*}}
Thus, every finite value can be attained as the $p$-adic valuation of $\left(\frac{F(x)}{F(y)}\right)$ for some $x,y\in\mathbb{Z}^{\left\lfloor\frac{n}{2}\right\rfloor}$. Using similar reasoning to the one from the proof of Theorem \ref{thm-new}, one can show that $R(F)$ is dense in $\mathbb{Q}_p$.
\end{rem}
 \begin{proof}[Proof of Theorem \ref{thm-5}] We have $C(x,y,z)=ax^3+by^3+cz^3$. If $p\neq 3$, then 
 by Lemma \ref{lem-7}, $C(x,y,z)$ has a non-singular solution $(x_0,y_0,z_0)$ modulo $p$. Suppose that $C(x_0,y_0,z_0)=0 \pmod{p}$  and $\frac{\partial C}{\partial x}(x_0,y_0,z_0)\neq 0\pmod{p}$. Then by Hensel's lemma, the  polynomial $f(x)=C(x,y_0,z_0)$ has a simple root in $\mathbb{Z}_p.$ Therefore, by Theorem \ref{thm6}, $R(f(\mathbb{N}))$  is dense in $\mathbb{Q}_p$. Hence, $R(C)\supset R(f(\mathbb{N})) $ is dense in $\mathbb{Q}_p$.
 \end{proof}
\begin{proof}[Proof of Corollary \ref{cor-1}]
	Write $a_i=p^{m_i}\ell_i$, $a_j=p^{m_j}\ell_j$ and $a_k=p^{m_k}\ell_k$, where $m_i\equiv m_j\equiv m_k\pmod{3}$ and $p\nmid\ell_i\ell_j\ell_k$. Without loss of generality assume that $m_i\geq m_j\geq m_k$. Put $\widetilde{C}(x_i,x_j,x_k)=a_ix_i^3+a_jx_j^3+a_kx_k^3$ and $\widehat{C}(x_i,x_j,x_k)=\ell_ix_i^3+\ell_jx_j^3+\ell_kx_k^3$. Then
 \begin{align*}
&\frac{a_i
	{\left(p^{\frac{m_k-m_i}{3}}x_i\right)}^3+a_j
	{\left(p^{\frac{m_k-m_j}{3}}x_j\right)}^3+a_k
	{x_k}^3 }{a_i
	{\left(p^{\frac{m_k-m_i}{3}}y_i\right)}^3+a_j
	{\left(p^{\frac{m_k-m_j}{3}}y_j\right)}^3+a_k
	{y_k}^3}
=\frac{\ell_ix_i^3+\ell_jx_j^3+\ell_kx_k^3}{\ell_iy_i^3+\ell_jy_j^3+\ell_ky_k^3},
 \end{align*}	
which means that $R(\widehat{C})\subset R(\widetilde{C}(\mathbb{Q}^3))=R(\widetilde{C}(\mathbb{Z}^3))=R(\widetilde{C})\subset R(C)$. From Theorem \ref{thm-5}, $R(\widehat{C})$ is dense in $\mathbb{Q}_p$. As a result, $R(C)$ is dense in $\mathbb{Q}_p$.
\end{proof}
\begin{proof}[Proof of Corollary \ref{cor-2}]
It suffices to note that there are $1\leq i<j<k\leq 7$ such that $\nu_p(a_i)\equiv\nu_p(a_j)\equiv\nu_p(a_k)\pmod{3}$. Then by Corollary \ref{cor-1}, $R(a_ix_i^3+a_jx_j^3+a_kx_k^3)$ is dense in $\mathbb{Q}_p$. Since $R(a_ix_i^3+a_jx_j^3+a_kx_k^3)\subset R(C)$, we see that $R(C)$ is dense in $\mathbb{Q}_p$.
\end{proof}
\begin{proof}[Proof of Theorem \ref{thm-4}]
	We claim that $\nu_p(F(\overline{x}))$ is a multiple of $n$ for all $\overline{x}$, where $\overline{x}=(x_1,x_2,\dots,x_r)\in \mathbb{Z}^r$. If $F(\overline{x})\not\equiv 0\pmod{p}$, then $\nu_p(F(\overline{x}))=0$. Suppose that $F(\overline{x})\equiv0\pmod{p}$. 
	Therefore $\overline{x}\equiv \overline{0}\pmod{p}$ as $F$ is anisotropic. Let $k=\min\{\nu_p(x_i)\colon i\in\{1,\dots,r\}\}$. Then, we get	$\nu_p(F(\overline{x}))=nk$
	since $F$ is anisotropic. Thus, by Lemma \ref{lem2}, $R(F)$ is not dense in  $\mathbb{Q}_p$.
\end{proof}


\end{document}